\documentclass[reqno,a4paper,12pt]{amsart}

\usepackage[pdfpagelabels]{hyperref}
\usepackage{amssymb}
\usepackage{url}

\newtheorem{theorem}{Theorem}%[section]
\newtheorem{lemma}[theorem]{Lemma}
\newtheorem{corollary}[theorem]{Corollary}
\newtheorem{proposition}[theorem]{Proposition}

\theoremstyle{definition}

\theoremstyle{remark}

\newcommand{\abs}[1]{\lvert#1\rvert}
\newcommand{\nm}[1]{\lVert#1\rVert}

\newcommand{\D}{\mathbb{D}}

\newcommand{\Z}{\mathbb{Z}}

\newcommand{\C}{\mathbb{C}}

\renewcommand{\phi}{\varphi}

\DeclareMathOperator{\dist}{dist}

\newcommand{\BMOA}{\rm BMOA}

\begin{document}

\title[Solutions having zeros on pre-given sequences]{Solutions of complex differential equation having zeros on pre-given sequences}
\thanks{The author is supported  in part by the Academy of Finland \#286877 and
the Finnish Academy of Science and Letters (Vilho, Yrj\"o and Kalle V\"ais\"al\"a Foundation).}

\author{Janne Gr\"ohn}
\address{Department of Physics and Mathematics, University of Eastern Finland\\ 
\indent P.O. Box 111, FI-80101 Joensuu, Finland}
\email{janne.grohn@uef.fi}

\date{\today}

% 34C10 Oscillation theory, zeros, disconjugacy and comparison theory
% 30D45 Bloch functions, normal functions, normal families

\begin{abstract}
Behavior of solutions of $f''+Af=0$ is discussed under the assumption
that $A$ is analytic in $\D$ and $\sup_{z\in\D}(1-|z|^2)^2|A(z)|<\infty$, where $\D$ is the unit disc of
the complex plane. As a main result it is shown that
such differential equation may admit a non-trivial solution whose zero-sequence does not satisfy the
Blaschke condition. This gives an answer to an open question in the literature.

It is also proved that $\Lambda\subset\D$ is the zero-sequence 
of a non-trivial solution of $f''+Af=0$
where $|A(z)|^2(1-|z|^2)^3\, dm(z)$ is a Carleson measure if and only if $\Lambda$ is uniformly separated.
As an application an~old result, according to which there exists a non-normal function
which is uniformly locally univalent, is improved.
\end{abstract}

\maketitle

%%%%%%%%%%%%%%%%%%%%%%%
%%%% ---- SECTION ---- %%%%
%%%%%%%%%%%%%%%%%%%%%%%

\section{Introduction}
 
Let $\mathcal{H}(\D)$ be the collection of analytic functions in the unit
disc $\D$ of the complex plane $\C$. This research concerns 
zero-sequences of non-trivial solutions ($f\not\equiv 0$) of 
\begin{equation} \label{eq:de2}
  f''+Af=0
\end{equation}
under the assumption $A\in H^\infty_2$, which means that
$A\in\mathcal{H}(\D)$ and $\nm{A}_{H^\infty_2} = \sup_{z\in\D} \,(1-|z|^2)^2 |A(z)|<\infty$.
In particular, we are interested in the following question:
\begin{itemize}
\item[\rm (Q)]
  Is it true that the zero-sequence $\{z_n\}$ of any non-trivial solution of \eqref{eq:de2}
  satisfies the Blaschke condition $\sum_n (1-|z_n|)<\infty$ if $A\in H^\infty_2$?
\end{itemize}

Question (Q) relates to so-called Blaschke-oscillatory equations, and appears in 
\cite[pp.~61--62]{H:2013}. Note that the characterization \cite[Lemma~3]{H:2013}
of Blaschke-oscillatory equations does not provide an immediate answer to (Q). 
Moreover, it is known that all non-trivial solutions of \eqref{eq:de2}
may lie outside the Nevanlinna class $N$, even if $A\in H^\infty_2$ \cite[pp.~57--58]{H:2013}.

We will make repeated use of \cite[Theorem~6.1]{L:1993}, which connects non-trivial 
solutions of \eqref{eq:de2} to a locally univalent meromorphic function in $\D$.
If $\nm{A}_{H^\infty_2}\leq 1$ then all non-trivial solutions of~\eqref{eq:de2}
vanish at most once in~$\D$ by \cite[Theorem~I]{N:1949}, and
if $\nm{A}_{H^\infty_2} >1$ then non-trivial solutions may have infinitely many zeros
\cite{H:1949}. The condition $A\in H^\infty_2$
is equivalent to the fact that zero-sequences of non-trivial solutions of \eqref{eq:de2} are
separated with respect to the pseudo-hyperbolic metric \cite[Theorems~3--4]{S:1955},
by a~constant depending on $\nm{A}_{H^\infty_2}$,
and hence zero-sequences \emph{almost} satisfy the Blaschke condition \cite[p.~162]{DS:2004}.
Many sufficient coefficient conditions implying 
an~affirmative answer to (Q) are known. For coefficient conditions placing all solutions 
of \eqref{eq:de2} to the Nevanlinna class, see \cite{H:2000, P:1982},
and for coefficient conditions placing all solutions to some Hardy space, see \cite{GHR:preprint,GNR:preprint,P:1982,R:2007}.
For a~more direct approach to zero-sequences of solutions of \eqref{eq:de2}, see \cite{GN:preprint}.

%%%%%%%%%%%%%%%%%%%%%%%
%%%% ---- SECTION ---- %%%%
%%%%%%%%%%%%%%%%%%%%%%%

\section{Results} \label{sec:results}

As the main result, we prove that there is $A\in H^\infty_2$
such that \eqref{eq:de2} admits a~non-trivial solution whose zero-sequence
does not satisfy the Blaschke condition. This answers (Q) in the negative.
Actually, we show that a non-trivial solution may vanish on any pre-given sequence of sufficiently small density.

We also obtain a complete description of zero-sequences of solutions of \eqref{eq:de2}
in the case that $A$ satisfies a condition stronger than $A\in H^\infty_2$.
In Section~\ref{subsec:normal}, we consider an application 
concerning normal meromorphic functions.

%%%%%%%%%%%%%%%%%%%%%%%%%%
%%%% ---- SUBSECTION ---- %%%%
%%%%%%%%%%%%%%%%%%%%%%%%%%

\subsection{Zero-sequences of solutions}

The sequence $\Lambda = \{z_n\}$ of points in $\D$ is said to be uniformly separated if
\begin{equation*}
  \inf_{k} \, \prod_{n\neq k} \, \left| \frac{z_n - z_k}{1-\overline{z}_n z_k} \right|>0,
\end{equation*}
while $\Lambda$ is called separated if
there exists a~constant $\delta=\delta(\Lambda)>0$ such that
$\varrho(z_n,z_k) = |z_n-z_k|/|1-\overline{z}_n z_k|>\delta$ for any $n\neq k$.
Unless otherwise stated, separation is understood 
with respect to the pseudo-hyperbolic metric.

Let $\Delta(\zeta,r) = \{ z\in\D : \varrho(z,\zeta)<r\}$ be the (open) pseudo-hyperbolic disc of radius $0<r<1$,
centered at $\zeta\in\D$, and let $n(\Lambda,\zeta,r)$ 
be the counting function for those points in $\Lambda$
which lie in $\Delta(\zeta,r)$.
The lower and upper uniform densities of $\Lambda$ are 
\begin{align}
  D^-(\Lambda) & = \liminf_{r\to 1^-} \, \left( \log\frac{1}{1-r} \right)^{-1} \, 
                 \inf_{\zeta\in\D} \int_0^r n(\Lambda,\zeta,s) \, ds,\label{eq:lowerden}\\
  D^+(\Lambda) & = \limsup_{r\to 1^-} \, \left( \log\frac{1}{1-r} \right)^{-1} \, 
                 \sup_{\zeta\in\D} \int_0^r n(\Lambda,\zeta,s) \, ds, \notag
\end{align}
respectively. For a comprehensive treatment of these densities,
and their connection to interpolation and sampling,
see \cite[Chapters~6--7]{DS:2004} and \cite[Chapter~3]{S:2004}.
See also the seminal papers \cite{S:1993, S:1993_2} by Seip.

%%%%%%%%%%%%%%%%%%%%%%%%
%%%% ---- THEOREM ---- %%%%
%%%%%%%%%%%%%%%%%%%%%%%%

\begin{theorem} \label{th:general}
If $\Lambda\subset\D$ is a separated sequence for which $D^+(\Lambda)<1$,
then there exists $A=A(\Lambda)\in H^\infty_2$ such that \eqref{eq:de2}
admits a non-trivial solution which vanishes at all points on~$\Lambda$.

Conversely, if $A\in H^\infty_2$ and $f$ is a non-trivial solution of \eqref{eq:de2}
whose zero-sequence is $\Lambda\subset\D$, then $\Lambda$ is separated, and
contains at most one point if $\nm{A}_{H^\infty_2}\leq 1$, while otherwise
\begin{equation*}
  D^+(\Lambda) \leq (2\pi+1) \! \left(1-\frac{2\nm{A}_{H^\infty_2}^{1/2}}{\nm{A}_{H^\infty_2}+1} \right)^{1/2}
  \!\!\Bigg( 1 - \left(1-\frac{2\nm{A}_{H^\infty_2}^{1/2}}{\nm{A}_{H^\infty_2}+1} \right)^{1/2} \, \Bigg)^{-2}.
\end{equation*}
\end{theorem}

Let~$\lesssim$ denote a one-sided estimate up to a constant and write~$\simeq$ for a two-sided estimate
up to constants. In the first part of Theorem~\ref{th:general} we construct a~solution $f$ of \eqref{eq:de2}
which vanishes on $\Lambda$ (but has other zeros also). 
The idea is to find an auxiliary function $g\in\mathcal{H}(\D)$ such that
\begin{equation} \label{eq:gasymp}
  (1-|z|^2)^\beta |g(z)| \simeq \varrho(z,\Lambda^\star), \quad z\in\D,
\end{equation}
for some constant $0<\beta<\infty$, and for some separated sequence $\Lambda^\star \supset \Lambda$
which satisfies $D^+(\Lambda^\star)<1$. Such function $g$ exists in the literature.
Then, $f=ge^{h}$ is a~solution of \eqref{eq:de2} for $A\in H^\infty_2$ 
if $h$ belongs to the Bloch space
and solves a certain interpolation problem.
Note that the additional assumption $D^+(\Lambda)<1$ in Theorem~\ref{th:general} holds for any
sequence $\Lambda\subset\D$ which is separated by a constant sufficiently close to one \cite[Lemma~9, p.~190]{DS:2004}.
The second part of Theorem~\ref{th:general} follows immediately
by combining \cite[Lemma~9, p.~190]{DS:2004}, \cite[Theorem~I]{N:1949} and \cite[Theorem~3]{S:1955}. It implies that
$D^+(\Lambda)\to 0^+$ at the linear rate as $\nm{A}_{H^\infty_2} \to 1^+$.

The following construction is due to Seip \cite[pp.~214--215]{S:1993}.
For $a>1$ and $b>0$, let $\Lambda=\Lambda(a,b)$ be the image of $\{a^j (bk+i)\}_{j,k\in\Z}$
under the Cayley transform (conformal map from the upper half-plane onto $\D$).
The set $\Lambda \subset \D$ satisfies $D^-(\Lambda)=D^+(\Lambda) = 2\pi/(b \log a)$ 
by \cite[p.~23]{S:1993_2}. In particular, there exists $g\in\mathcal{H}(\D)$
which satisfies \eqref{eq:gasymp} for $\Lambda^\star=\Lambda$ and $\beta=2\pi/(b \log a)$.
Now $\Lambda$ is a separated sequence which behaves (essentially) as badly as possible
in terms of the Blaschke condition; recall that the lower uniform density of any separated Blaschke sequence
is zero by \eqref{eq:lowerden}.
As a~straightforward consequence of Theorem~\ref{th:general}, we obtain:

%%%%%%%%%%%%%%%%%%%%%%%%%%
%%%% ---- COROLLARY ---- %%%%
%%%%%%%%%%%%%%%%%%%%%%%%%%

\begin{corollary} \label{cor:general}
Let $\Lambda=\Lambda(a,b) \subset\D$ be as above, and let $2\pi/(b \log a)<1$.
Then, there exists $A=A(a,b) \in H^\infty_2$ such that \eqref{eq:de2}
admits a~non-trivial solution whose zero-sequence is $\Lambda$, and hence does not satisfy the Blaschke condition.
\end{corollary}

Let $0<p<\infty$. A positive Borel measure $\mu$ on $\D$ is called a (bounded) $p$-Carleson measure
provided that
\begin{equation*}
\sup_{I\subset \partial\D} \, \frac{\mu( Q(I) ) }{|I|^p} < \infty,
\quad
Q(I) = \big\{ re^{i\theta}  :  1-|I| \leq r < 1, \, e^{i\theta}\in I \big\},
\end{equation*}
where the supremum is taken over all subarcs $I\subset \partial\D$ and $|I|$ denotes the
length of $I$ (normalized so that $|\partial\D|=1$). 
These measures can be described in conformally
invariant terms \cite[Lemma~2.1]{ASX:1996}. In fact, the positive Borel measure $\mu$ is 
a~$p$-Carleson measure if and only~if
\begin{equation} \label{eq:carleson}
  \sup_{a\in\D} \, \int_\D \left( \frac{1-\abs{a}^2}{|1-\overline{a}z|^2} \right)^p d\mu(z) < \infty.
\end{equation}
For $p=1$ the condition \eqref{eq:carleson} characterizes the classical Carleson measures,
which were invented to study interpolation by bounded analytic functions.
See \cite{G:2007} for a~general reference.

There are two types of measures which play a~role in this study. First, let 
$\delta_{z}$ be the Dirac mass at the point $z\in\D$.
We consider separated sequences $\Lambda\subset\D$ for which 
\begin{equation} \label{eq:seqcarleson}
\sum_{z_n\in\Lambda} (1-|z_n|)^p\, \delta_{z_n}
\end{equation}
is a $p$-Carleson measure. 
Such sequences are uniformly separated for any $0<p\leq1$. Second, we
consider functions $A\in\mathcal{H}(\D)$ for which
\begin{equation} \label{eq:muap}
d\mu_{A,p}(z)=|A(z)|^2(1-|z|^2)^{2+p}\, dm(z)
\end{equation}
is a~$p$-Carleson measure. Here $dm(z)$ is the Lebesgue area measure on $\D$.
We write $\mu_A = \mu_{A,1}$ for short.
Such functions satisfy $A\in H^\infty_2$ by the subharmonicity of~$|A|^2$.
The effect of the parameter~$p$ is more evident when second primitives of $A\in\mathcal{H}(\D)$ are considered
\cite[Theorem~3.2]{R:2003}:
\eqref{eq:muap} is a $p$-Carleson measure if and only if the second primitive of $A$ belongs to
$Q_p$. The space $Q_p$ consists of those $f\in\mathcal{H}(\D)$ for which $|f'(z)|^2(1-|z|^2)^p\, dm(z)$
is a $p$-Carleson measure. For $1<p<\infty$, $Q_p$ is the Bloch space while 
$Q_{p_1}\subsetneq Q_{p_2} \subsetneq Q_1=\BMOA$ for any $0<p_1<p_2<1$. 
See \cite{ASX:1996,EX:1997,NX:1997} and the references therein, for more details.

%%%%%%%%%%%%%%%%%%%%%%%%%
%%%% ---- THEOREM ----- %%%%
%%%%%%%%%%%%%%%%%%%%%%%%%

\begin{theorem} \label{th:qp}
Let $0<p\leq 1$. If $\Lambda\subset\D$ is a separated sequence such that \eqref{eq:seqcarleson} is
a~$p$-Carleson measure, then there exists $A=A(\Lambda)\in\mathcal{H}(\D)$ such that
\eqref{eq:muap} is a $p$-Carleson measure and \eqref{eq:de2} admits a non-trivial
solution whose zero-sequence is $\Lambda$.
\end{theorem}

Theorems~\ref{th:general} and \ref{th:qp} improve \cite[Corollary~7]{GH:2012}, which states that
any uniformly separated sequence can appear as the zero-sequence of a non-trivial solution
of \eqref{eq:de2} where $A\in H^\infty_2$, i.e., the second primitive of $A$
is in the Bloch space. Theorem~\ref{th:general} shows that we can prescribe zero-sequences of 
strictly positive uniform density 
under the same coefficient condition while Theorem~\ref{th:qp} implies that 
any sufficiently separated sequence can be prescribed
such that the second primitive of $A$ belongs to $Q_p$ for fixed $0<p\leq 1$.
When prescribing infinite zero-sequences to non-trivial solutions of \eqref{eq:de2}, we cannot
expect that the coefficient $A$ is even close to be bounded.
The breaking point lies inside~$H^\infty_2$: if $A\in\mathcal{H}(\D)$ and there exists $0<R<1$ such that
$(1-|z|^2)^2|A(z)| \leq 1$ for $R<|z|<1$, then all non-trivial solutions of \eqref{eq:de2} have at most
finitely many zeros in $\D$ \cite[Theorem~1]{S:1955}.

By Theorem~\ref{th:qp}, we obtain a~complete description of 
zero-sequences of non-trivial solutions of \eqref{eq:de2} in the case that 
$d\mu_A(z)=|A(z)|^2(1-|z|^2)^3\, dm(z)$ is a Carleson measure.

%%%%%%%%%%%%%%%%%%%%%%%%%%%
%%%% ---- COROLLARY ----- %%%%
%%%%%%%%%%%%%%%%%%%%%%%%%%%

\begin{corollary} \label{cor:desc}
A sequence $\Lambda\subset\D$ is the zero-sequence of a non-trivial solution of~\eqref{eq:de2}
where $\mu_A$ is a Carleson measure if and only if $\Lambda$ is uniformly separated.
\end{corollary}

In Corollary~\ref{cor:desc}, all zero-sequences are uniformly separated by
\cite[Corollary~3]{GN:preprint}. The converse assertion follows from Theorem~\ref{th:qp} by taking $p=1$.
The following observation concerns the case $A\in H^\infty_{2,0}$, which means that
$A\in\mathcal{H}(\D)$ and $\lim_{|z|\to 1^{-}} \, (1-|z|^2)^2 |A(z)| = 0$.
Sequence $\Lambda\subset\D$ is the zero-sequence of a non-trivial solution of~\eqref{eq:de2}
where $A\in H^\infty_{2,0}$ if and only if $\Lambda$ is a~finite sequence of distinct points in $\D$.
The fact that all zero-sequences are finite follows
from \cite[Theorem~1]{S:1955}, while the converse assertion is proved by constructing
a non-trivial solution of \eqref{eq:de2}, which has finitely many prescribed zeros \cite[Section~10]{H:2013}.

%%%%%%%%%%%%%%%%%%%%%%%%%%
%%%% ---- SUBSECTION ---- %%%%
%%%%%%%%%%%%%%%%%%%%%%%%%%

\subsubsection*{Conformally invariant collections of zero-sequences}
Let $X$ be a space of analytic functions, and 
let $\mathcal{Z}(X)$ be the collection of sequences $\Lambda\subset \D$ 
for which there exists $A=A(\Lambda)\in X$ such that \eqref{eq:de2} 
admits a~non-trivial solution whose zero-sequence is (precisely)~$\Lambda$.
Some parts of the following result are known in another form, see the proof for references.

%%%%%%%%%%%%%%%%%%%%%%%%
%%%% ---- THEOREM ---- %%%%
%%%%%%%%%%%%%%%%%%%%%%%%

\begin{proposition} \label{prop:1c}
The following statements hold:
\begin{enumerate}
\item[\rm (a)]
$\mathcal{Z}(H^\infty_2)$ is conformally invariant.

\item[\rm (b)]
If $\Lambda$ is separated and $D^+(\Lambda)<1$, then $\Lambda \subset\Lambda'$
for some $\Lambda' \in \mathcal{Z}(H^\infty_2)$. Conversely,
if $\Lambda\in \mathcal{Z}(H^\infty_2)$, then $\Lambda$ is separated and $D^+(\Lambda)<\infty$.

\item[\rm (c)]
$\mathcal{Z}(H^\infty_2)$ contains non-Blaschke sequences. However,
if $\Lambda\in \mathcal{Z}(H^\infty_2)$ and
\begin{equation} \label{eq:dist}
  \int_0^{2\pi} \log\dist(e^{i\theta},\Lambda)\,d\theta > -\infty,
\end{equation}
where $\dist$ is the Euclidean distance, then $\Lambda$ is a Blaschke sequence.
\end{enumerate}
\end{proposition}

Let $K$ be the space which consists of the second derivatives of $\BMOA$ functions. 
Consequently, $\mathcal{Z}(K)$ is the collection of zero-sequences of non-trivial solutions of \eqref{eq:de2} 
induced by those coefficients $A\in\mathcal{H}(\D)$ for which $d\mu_A(z)= |A(z)|^2(1-|z|^2)^3\, dm(z)$ 
is a~Carleson measure. By Corollary~\ref{cor:desc}, 
$\Lambda\in \mathcal{Z}(K)$ if and only if $\Lambda$ is uniformly separated.
The following observations follow from 
known properties of uniformly separated sequences:
\begin{enumerate}
\item[\rm (a)]
  $\mathcal{Z}(K)$ is conformally invariant.

\item[\rm (b)]
  If $\Lambda \subset \Lambda'$ for some $\Lambda'\in\mathcal{Z}(K)$,
  then $\Lambda\in\mathcal{Z}(K)$.

\item[\rm (c)]
  If $\Lambda_1,\Lambda_2\in\mathcal{Z}(K)$, and
  $\Lambda_1 \cup \Lambda_2$ is separated, then $\Lambda_1 \cup \Lambda_2 \in \mathcal{Z}(K)$.

\item[\rm (d)]
  If $\Lambda=\{z_n\} \in \mathcal{Z}(K)$, and $\Lambda'=\{ z_n' \}\subset \D$ is a sequence
  such that $\sup_n  \varrho(z_n,z_n')$ is sufficiently small, then
  $\Lambda' \in \mathcal{Z}(K)$.

\item[\rm (e)]
  $\mathcal{Z}(K)$ contains only Blaschke sequences. However, there
  are separated Blaschke sequences which are not in $\mathcal{Z}(K)$.
\end{enumerate}

By subharmonicity $K\subset H^\infty_2$, and hence $\mathcal{Z}(K)\subset \mathcal{Z}(H^\infty_2)$.
It is curious that $\mathcal{Z}(H^\infty_{2,0})\subset \mathcal{Z}(K)$ even though
there exist functions $A\in H^\infty_{2,0}$ for which $\mu_A$ is not a Carleson measure;
typical examples of such functions are given in terms of lacunary series.
Note that $\mathcal{Z}(H^\infty_{2,0})$ is the collection of  
finite sequences of distinct points in $\D$.

%%%%%%%%%%%%%%%%%%%%%%%%%%
%%%% ---- SUBSECTION ---- %%%%
%%%%%%%%%%%%%%%%%%%%%%%%%%

\subsection{Normal functions} \label{subsec:normal}

A function $w$ meromorphic in the unit disc $\D$ is said to be normal if
%\begin{equation*}
  $\sup_{z\in\D} \, (1-|z|^2) \, w^{\#}(z) < \infty$,
%\end{equation*}
where $w^{\#} = |w'|/(1+|w|^2)$ is the spherical derivative of $w$. 
Actually, a~meromorphic function $w$ is normal if and only if
$\{ w \circ \phi : \text{$\phi$ conformal automorphism of $\D$}\}$ is a normal family in $\D$ (in
the sense of Montel). For more details, see \cite{LV:1957}.

In \cite{L:1973}, Lappan gives an answer to a question of Hayman by showing 
that there exists a non-normal $w\in \mathcal{H}(\D)$ whose Schwarzian derivative
\begin{equation*}
  S_w = \left( \frac{w''}{w'} \right)' - \frac{1}{2} \left( \frac{w''}{w'} \right)^2
\end{equation*}
belongs to $H^\infty_2$. In a subsequent paper \cite[Theorem~5]{L:1978}
a concrete function having these properties is presented.
Recall that, for $w$ meromorphic in $\D$, $S_w\in H^\infty_2$ if and only if $w$ is uniformly locally univalent;
see for example \cite[Lemma~B]{G:2016} and the references therein. 
As a consequence of the proof of Theorem~\ref{th:qp}, we obtain:

%%%%%%%%%%%%%%%%%%%%%%%%%%
%%%% ---- COROLLARY ---- %%%%
%%%%%%%%%%%%%%%%%%%%%%%%%%

\begin{corollary} \label{cor:normal}
For any $0<p\leq 1$, 
there exists a non-normal meromorphic function $w$ in $\D$ such that
$S_w\in\mathcal{H}(\D)$ and $|S_w(z)|^2(1-|z|^2)^{2+p}\, dm(z)$ is a~$p$-Carleson measure.
\end{corollary}

In Corollary~\ref{cor:normal}, $S_w\in H^\infty_2$ by the subharmonicity of $|S_w|^2$. 
By construction, the function $w$ has prescribed separated poles $\Lambda$
such that \eqref{eq:seqcarleson} is a $p$-Carleson measure, $w$ belongs to the Nevanlinna class 
of meromorphic functions and emerges  as a~primitive of $1/f^2$ where $f\in Q_p \cap H^\infty$ 
is a solution of \eqref{eq:de2} for $A=S_w/2$.

%%%%%%%%%%%%%%%%%%%%%%%
%%%% ---- SECTION ---- %%%%
%%%%%%%%%%%%%%%%%%%%%%%

\section{Proofs of the results}

The growth space $H^\infty_\alpha$, for $0\leq \alpha < \infty$, consists of those
$g\in\mathcal{H}(\D)$ for~which
\begin{equation*}
  \nm{g}_{H^\infty_\alpha} = \sup_{z\in\D} \, (1-|z|^2)^\alpha |g(z)|<\infty.
\end{equation*}
In particular, $H^\infty = H^\infty_0$. Before the proof of Theorem~\ref{th:general}, we consider
an~auxiliary result which resembles the classical Schwarz lemma. See also \cite[Lemma~7, p.~209]{DS:2004}.
The proof of Lemma~\ref{lemma:schwarz} is presented for convenience of the reader.

%%%%%%%%%%%%%%%%%%%%%%
%%%% ---- LEMMA ---- %%%%
%%%%%%%%%%%%%%%%%%%%%%

\begin{lemma} \label{lemma:schwarz}
Let $g\in H^\infty_\alpha$ where $0\leq \alpha<\infty$, and let $0<\delta<1$. If $g(z_0)=0$
for some $z_0\in\D$, then there exists a positive constant $C=C(\alpha, \delta)$ such that
\begin{equation} \label{eq:ggest}
  |g(z)| \leq \frac{C \, \nm{g}_{H^\infty_\alpha}\, \varrho(z,z_0)}{(1-|z_0|^2)^\alpha}, \quad z\in \Delta(z_0,\delta).
\end{equation}
\end{lemma}

%%%%%%%%%%%%%%%%%%%%%%
%%%% ---- PROOF ---- %%%%
%%%%%%%%%%%%%%%%%%%%%%

\begin{proof}
Consider the function $h(z) = g(z)/((z_0-z)/(1-\overline{z}_0 z))$, and note that $h\in\mathcal{H}(\D)$.
There exists a positive constant $C^\star=C^\star(\alpha,\delta)$ such that
\begin{equation*}
  |h(z)| \leq \frac{|g(z)|}{\varrho(z,z_0)} \leq \frac{\nm{g}_{H^\infty_\alpha}}{\delta\, (1-|z|^2)^\alpha}
  \leq \frac{C^\star \, \nm{g}_{H^\infty_\alpha}}{\delta\, (1-|z_0|^2)^\alpha}, 
  \quad z\in\partial \Delta(z_0,\delta).
\end{equation*}
By the maximum modulus principle this inequality holds for all $z\in\Delta(z_0,\delta)$,
which implies \eqref{eq:ggest} for $C=C^\star/\delta$.
\end{proof}

%%%%%%%%%%%%%%%%%%%%%%
%%%% ---- PROOF ---- %%%%
%%%%%%%%%%%%%%%%%%%%%%

\begin{proof}[Proof of Theorem~\ref{th:general}]
Let $\Lambda$ be a separated sequence for which $D^+(\Lambda)<1$.
By \cite[Theorem~5, p.~220]{DS:2004}, there exist a separated sequence $\Lambda^\star \supset \Lambda$
and $g\in\mathcal{H}(\D)$ such that \eqref{eq:gasymp} holds for $\beta=(D^+(\Lambda)+1)/2<1$.
According to \cite[Lemma~19, p.~235]{DS:2004}, we have $D^-(\Lambda^\star) = D^+(\Lambda^\star) = \beta$.
Let $f=g e^h$ where $h\in\mathcal{H}(\D)$ is defined later. 
The function $f$ is a~solution of \eqref{eq:de2} with $A\in\mathcal{H}(\D)$,
\begin{equation*} 
  A = -\frac{f''}{f} = - \frac{g''+2g'h'}{g} - (h')^2 -h'',
\end{equation*}
provided that $h\in\mathcal{H}(\D)$ and $h'$ satisfies the interpolation property 
\begin{equation} \label{eq:interpolation}
  h'(z_n)= - \frac{1}{2} \, \frac{g''(z_n)}{g'(z_n)}, \quad z_n\in\Lambda^\star.
\end{equation}
In particular, because of \eqref{eq:interpolation}, $A$ has a removable singularity at each zero of~$g$.
Since $g''\in H^\infty_{\beta+2}$ by Cauchy's integral formula, and
\begin{equation} \label{eq:ope} 
  \inf_{z_n\in\Lambda^\star} \, (1-|z_n|^2)^{\beta+1} |g'(z_n)| >0
\end{equation}
by \eqref{eq:gasymp}, we deduce
\begin{equation} \label{eq:blochest}
  \sup_{z_n\in\Lambda^\star} \, (1-|z_n|^2)  \left| \frac{g''(z_n)}{g'(z_n)} \right| < \infty.
\end{equation}
Since $D^+(\Lambda^\star)=\beta<1$, \cite[Theorem~1.2]{S:1993_2} and \eqref{eq:blochest}
imply that there exists $h\in\mathcal{H}(\D)$ such that $h'\in H^\infty_1$ and the interpolation
property \eqref{eq:interpolation} holds.

It remains to prove that $A\in H^\infty_2$.
Let $0<\delta<1$ be a sufficiently small constant such that the pseudo-hyperbolic discs
$\Delta(z_n,\delta)$ are pairwise disjoint for $z_n\in\Lambda^\star$. On one hand,
by \eqref{eq:gasymp} and $h'\in H^\infty_1$, we obtain
\begin{equation*}
  \sup \bigg\{ (1-|z|^2)^2 |A(z)| \, : \,  z\in \D \setminus \bigcup_{z_n\in\Lambda^\star} \Delta(z_n,\delta) \bigg\} < \infty.
\end{equation*}
On the other hand, if $z\in\Delta(z_n,\delta)$ for some $z_n\in\Lambda^\star$, then
we apply Lemma~\ref{lemma:schwarz} to $g''+2g'h' \in H^\infty_{\beta+2}$ (which vanishes at all points $z_n\in\Lambda^\star$
by the interpolation property)
to deduce that $(1-|z|^2)^2|A(z)|$ is uniformly bounded also for any $z\in \bigcup_{z_n\in\Lambda^\star} \Delta(z_n,\delta)$.
This completes the proof of Theorem~\ref{th:general}.
\end{proof}

The proof of Theorem~\ref{th:general} produces a~\emph{non-normal} solution of \eqref{eq:de2} 
under the restriction $A\in H^\infty_2$. See \cite[Theorem~3]{G:2016} for another example. 
To show that $f=ge^h$ in the proof of Theorem~\ref{th:general} is non-normal, we argue as follows.
For any $\xi\in\partial\D$, there exists a subsequence 
$\Lambda' = \Lambda'(\xi)$ of $\Lambda^\star$ which converges non-tangentially to $\xi$. 
This follows from \cite[Corollary, p.~188]{DS:2004} as $D^-(\Lambda^\star)>0$.
The Makarov law of the iterated logarithm \cite[Theorem~8.10]{P:1992} gives
\begin{equation} \label{eq:makarov}
  \limsup_{r\to 1^-} \, \frac{|h(r\xi)|}{\Psi(r)} \leq \nm{h'}_{H^\infty_1}, \quad \Psi(r) 
  = \sqrt{\log\frac{1}{1-r} \log\log\log\frac{1}{1-r}},
\end{equation}
for almost every $\xi\in\partial\D$.
Fix $\xi\in\partial\D$ such that \eqref{eq:makarov} holds, and let $\Lambda' = \Lambda'(\xi)$
be the corresponding subsequence of $\Lambda^\star$. By \cite[Proposition~1, p.~43]{DS:2004},
\begin{equation} \label{eq:dest}
  | h(z_n) | \leq \big| h(|z_n| \xi) \big| + \frac{\nm{h'}_{H^\infty_1}}{2}\, 
  \log\frac{1+\varrho(z_n,|z_n| \xi)}{1-\varrho(z_n,|z_n| \xi)},
  \quad z_n\in\Lambda',
\end{equation}
where $\sup\{ \varrho(z_n,|z_n| \xi) : z_n\in\Lambda' \} < 1$.
If $z_n\in \Lambda'$ and $|z_n|$ is sufficiently close to one, then
\begin{align*}
  (1-|z_n|^2) \, |f'(z_n)| 
  & = (1-|z_n|^2) |g'(z_n)| \, \exp\!\big(\text{Re\ } h(z_n)\big)\\
  & \gtrsim  \frac{1}{(1-|z_n|^2)^\beta} \, \exp\!\Big( - \big(\nm{h'}_{H^\infty_1}+1\big) \,\Psi(|z_n|) \Big)
\end{align*}
by \eqref{eq:ope}, \eqref{eq:makarov} and \eqref{eq:dest}.
It follows that $\sup_{z_n\in\Lambda^\star} \, (1-|z_n|^2) \, |f'(z_n)| = \infty$, and 
hence $f$ is non-normal by \cite[Proposition~7]{GNR:preprint}.

%%%%%%%%%%%%%%%%%%%%%%
%%%% ---- PROOF ---- %%%%
%%%%%%%%%%%%%%%%%%%%%%

\begin{proof}[Proof of Theorem~\ref{th:qp}]
Let $0<p\leq 1$, and let $\Lambda\subset\D$ be any separated sequence such that \eqref{eq:seqcarleson}
is a $p$-Carleson measure. If $B$ is a Blaschke product whose zero-sequence is $\Lambda$, then
$B\in Q_p \cap H^\infty$ by \cite[Theorem~2.2]{EX:1997}; the case $p=1$ is of course trivial,
since $Q_1 \cap H^\infty = H^\infty$. Now
\begin{equation*}
\inf_{z_n\in\Lambda} \, (1-|z_n|^2)|B'(z_n)|
 = \inf_{z_n\in\Lambda} \, \prod_{z_k\neq z_n} \left| \frac{z_k-z_n}{1-\overline{z}_k z_n} \right| >0
\end{equation*}
by the uniform separation of $\Lambda$, and hence
\begin{equation*}
\sup_{z_n\in\Lambda} \, \frac{\left| B''(z_n) \right|}{\left| B'(z_n) \right|^2} < \infty.
\end{equation*}
Define $f=Be^h$, where $h=Bk$ and 
$k\in\mathcal{H}(\D)$ is a solution to the interpolation problem 
\begin{equation*}
k(z_n) = -\frac{B''(z_n)}{2 \left( B'(z_n) \right)^2}, \quad z_n\in\Lambda. 
\end{equation*}
By \cite[Theorem~1.3]{NX:1997}, where the condition (b) follows from the fact that
\eqref{eq:seqcarleson} is a~$p$-Carleson measure,
or \cite[Theorem~3]{C:1962} if $p=1$, 
we may assume that $k\in Q_p\cap H^\infty$.
Then, as in the proof of Theorem~\ref{th:general}, it follows that 
$f$ is a~non-trivial solution of \eqref{eq:de2} where $A\in\mathcal{H}(\D)$ and 
\begin{equation} \label{eq:expanded2} 
  A = -\frac{f''}{f} = - \frac{B''+2B'h'}{B} - (h')^2 -h''.
\end{equation}
Since $Q_p\cap H^\infty$ is an~algebra, we have
$h\in Q_p\cap H^\infty$ and hence $f\in Q_p\cap H^\infty$.

It remains to prove that \eqref{eq:muap} is a $p$-Carleson measure.
Let $0<\delta<1$ be a sufficiently small constant such that the pseudo-hyperbolic discs
$\Delta(z_n,\delta)$ are pairwise disjoint for $z_n\in\Lambda$. 
We proceed to verify \eqref{eq:carleson} for $\mu=\mu_{A,p}$ in two parts. 
Denote $\Omega = \, \D \setminus \, \bigcup_{z_n\in\Lambda} \Delta(z_n,\delta)$.
Since $|B|$ is uniformly bounded away from zero on $\Omega$ and $h\in H^\infty$, 
we obtain
\begin{equation*} 
  |A(z)| = \frac{|f''(z)|}{|f(z)|} \lesssim |f''(z)| \, \big| e^{-h(z)} \big| \lesssim |f''(z)|,  \quad z\in \Omega.
\end{equation*}
Consequently,
\begin{align*} 
  & \sup_{a\in\D} \, \int_\Omega \left( \frac{1-\abs{a}^2}{|1-\overline{a}z|^2} \right)^p  d\mu_{A,p}(z) \\
 & \qquad \lesssim \sup_{a\in\D} \, \int_{\D} |f''(z)|^2(1-|z|^2)^{2+p} \left( \frac{1-\abs{a}^2}{|1-\overline{a}z|^2} \right)^p dm(z)
< \infty
\end{align*}
as $f\in Q_p$; see \cite[Theorem~3.2]{R:2003}.
Only the term $-(B''+2B'h')/B$ in \eqref{eq:expanded2} brings us additional trouble 
on $\D \setminus \Omega$, and hence it suffices to show that
\begin{equation*} 
  I = \sup_{a\in\D} \, \sum_{z_n\in\Lambda} \, \int_{\Delta(z_n,\delta)} 
  \left| \frac{B''(z)+2B'(z)\, h'(z)}{B(z)} \right|^2 \, \frac{(1-|z|^2)^{2+p} (1-\abs{a}^2)^p}{|1-\overline{a}z|^{2p}} \, dm(z)
\end{equation*}
is finite. Since $|B(z)|\gtrsim \varrho(z,z_n)$ for $z\in \Delta(z_n,\delta)$, 
and $|1-\overline{a}z| \simeq |1-\overline{a}z_n|$
for $z\in \Delta(z_n,\delta)$ and $a\in\D$ (with comparison constants
independent of $a$), 
\begin{equation*}
  I \lesssim \, \sup_{a\in\D} \, \sum_{z_n\in\Lambda} \frac{(1-|z_n|^2)^p(1-\abs{a}^2)^p}{|1-\overline{a}z_n|^{2p}} < \infty
\end{equation*}
by  Lemma~\ref{lemma:schwarz} and the fact that \eqref{eq:seqcarleson} is a $p$-Carleson measure. This 
completes the proof of Theorem~\ref{th:qp}.
\end{proof}

We may apply the Corona theorem for the algebra $Q_p\cap H^\infty$ to sharpen a~property in \cite{GHR:preprint}
(corresponding to the case $p=1$).
Let $0<p<1$. Assume that $f_1,f_2 \in Q_p \cap H^\infty$ are linearly independent solutions of \eqref{eq:de2} such that
\begin{equation} \label{eq:corona}
\inf_{z\in\D} \big( |f_1(z)| + |f_2(z)| \big) >0.
\end{equation}
By \cite[Theorem~1.1]{NX:1997} there exist $g_1,g_2\in Q_p \cap H^\infty$ such that $f_1g_1+f_2g_2 \equiv 1$.
Differentiate twice and apply \eqref{eq:de2} to conclude $A=f_1g_1'' + f_2 g_2'' + 2 (f_1'g_1'+f_2'g_2')$. 
Hence $d\mu_{A,p}(z) = |A(z)|^2 (1-|z|^2)^{2+p}\, dm(z)$ is a $p$-Carleson measure.
% by standard estimates.

To construct an example where \eqref{eq:corona} holds, we consider a method from \cite[p.~58]{H:2013}. 
Let $0<p<1$, and choose $g\in Q_p \cap H^\infty$. Define $h\in\mathcal{H}(\D)$ such that $h'=e^{-2g}$, 
which implies that $h\in Q_p \cap H^\infty$ and $h''+2g'h'=0$.
Then, functions $e^{g+h}, e^{g-h} \in Q_p \cap H^\infty$ are zero-free linearly independent solutions of \eqref{eq:de2} where
$A=-g''-(g')^2-(h')^2\in \mathcal{H}(\D)$. Estimate \eqref{eq:corona} follows from the fact that both solutions are uniformly bounded
away from zero. In this case it is easy to verify that $d\mu_{A,p}$ is a~$p$-Carleson measure.

%%%%%%%%%%%%%%%%%%%%%%
%%%% ---- PROOF ---- %%%%
%%%%%%%%%%%%%%%%%%%%%%

\begin{proof}[Proof of Proposition~\ref{prop:1c}]
(a) Let $\Lambda\in\mathcal{Z}(H^\infty_2)$ and let $\tau$ be a conformal automorphism of~$\D$.
We need to prove that $\tau(\Lambda)\in\mathcal{Z}(H^\infty_2)$.
By assumption, there exists $A=A(\Lambda)\in H^\infty_2$ such that
\eqref{eq:de2} admits a~non-trivial solution $f$ whose zero-sequence is $\Lambda$.
Let $\phi$ be the inverse of~$\tau$. Consequently,
$g=(f \circ \varphi )\, (\varphi')^{-1/2}$ is a non-trivial solution of
\begin{equation*}
g''+Bg=0, \quad B=(A\circ \varphi) (\varphi')^2,
\end{equation*}
see \cite[Lemma~1]{P:1982}.
Since $g$ vanishes precisely on  $\tau(\Lambda)$,
and $\nm{B}_{H^\infty_2} =\nm{A}_{H^\infty_2}$ by standard estimates, 
we have $\tau(\Lambda)\in\mathcal{Z}(H^\infty_2)$. 

Note that (b) and the first part of (c) follow from Theorem~\ref{th:general}, \cite[Theorem~3]{S:1955}
and Corollary~\ref{cor:general}. Hence, it suffices to prove the second part of~(c).
Suppose that $\Lambda\in \mathcal{Z}(H^\infty_2)$
and \eqref{eq:dist} holds. Consequently, there exists a coefficient $A=A(\Lambda) \in H^\infty_2$
such that \eqref{eq:de2} admits a non-trivial solution $f$ whose zero-sequence is~$\Lambda$.
By \cite[Example~1]{P:1982}, $f\in H^\infty_\alpha$ for any sufficiently
large $\alpha=\alpha(\nm{A}_{H^\infty_2})<\infty$. According to \eqref{eq:dist} and
\cite[Theorem, p.~146]{K:2006}, $\Lambda$ is a~Blaschke set of $H^\infty_\alpha$,
and hence $\Lambda$ is a~Blaschke sequence.
\end{proof}

%%%%%%%%%%%%%%%%%%%%%%
%%%% ---- PROOF ---- %%%%
%%%%%%%%%%%%%%%%%%%%%%

\begin{proof}[Proof of Corollary~\ref{cor:normal}]
Let $\Lambda$ be a separated sequence having 
infinitely many points such that \eqref{eq:seqcarleson} is a $p$-Carleson measure, 
and let $f$ be the function in the proof of Theorem~\ref{th:qp}.
In particular, $f=Be^h$ is a non-trivial solution of \eqref{eq:de2}
where $A\in\mathcal{H}(\D)$ and \eqref{eq:muap} is a~$p$-Carleson measure.
Here $B\in Q_p \cap H^\infty$ is the Blaschke product corresponding to $\Lambda$ and $h\in Q_p \cap H^\infty$.

Let~$g$ be a solution of \eqref{eq:de2} which is
linearly independent to $f$. We may assume that the Wronskian determinant 
satisfies $W(f,g) = f g' - f'g\equiv 1$.
If we define $w=g/f$, then $w$ is a locally univalent 
meromorphic function in $\D$ such that $S_w=2A$ and $w' = 1/f^2$. Consequently,
$S_w\in\mathcal{H}(\D)$ and $|S_w(z)|^2(1-|z|^2)^{2+p}\, dm(z)$ is a $p$-Carleson measure.
It remains to show that $w$ is non-normal. Since $\Lambda$ is uniformly separated,
there exists a constant $0<\delta<1$ such that
$\delta <  (1-|z_n|^2)|B'(z_n)| \leq 1$ for all $z_n\in\Lambda$, and hence
\begin{equation*}
  |g(z_n)| = \frac{1}{|f'(z_n)|} = \frac{|e^{-h(z_n)}|}{|B'(z_n)|} \simeq 1-|z_n|^2, \quad z_n\in\Lambda.
\end{equation*}
We conclude that
\begin{equation*}
  (1-|z_n|^2)\, w^{\#}(z_n) = \frac{1-|z_n|^2}{|f(z_n)|^2 + |g(z_n)|^2} \simeq \frac{1}{1-|z_n|^2}, \quad z_n\in\Lambda,
\end{equation*}
which means that $w$ is non-normal. Finally, we point out that $w$ belongs to the Nevanlinna class
of meromorphic functions by \cite[Corollary~3]{GN:preprint} and \cite[Lemma~3]{H:2013}.
\end{proof}

%%%%%%%%%%%%%%%%%%%%%%%
%%%% ---- SECTION ---- %%%%
%%%%%%%%%%%%%%%%%%%%%%%

\section*{Acknowledgements}

The author thanks A.~Nicolau for helpful conversations, and
gratefully acknowledges the hospitality of \emph{Departament de Matem\`atiques}, 
\emph{Universitat Aut\`onoma de Barcelona}.

%%%%%%%%%%%%%%%%%%%%%%%%%%%%%%%%
%%%% ---- BIBLIOGRAPHY ---- %%%%
%%%%%%%%%%%%%%%%%%%%%%%%%%%%%%%%

\end{document}